\documentclass[12pt]{amsart}
\usepackage[usenames,dvipsnames]{xcolor}
\usepackage{tikz}
\usepackage{fancyhdr}
\usepackage{txfonts}
\usepackage{graphicx}
\usepackage{epsfig}
\usepackage{mathrsfs}
\usepackage{amssymb}
\usepackage{latexsym}
\usepackage{amsmath}
\usepackage{amsfonts}
\usepackage{amsbsy}
\usepackage{amscd}
\usepackage{amsthm}
\usepackage{subfigure}
\usepackage{verbatim}
\usepackage{color,tikz}
\usetikzlibrary{calc}

\usepackage{amsfonts}
\usepackage{amsbsy,calc}
 \usepackage{graphicx,ifthen,cite}

\newtheorem{prop}{Proposition}[section]
\newtheorem{lem}[prop]{Lemma}

\newtheorem{defi}[prop]{Definition}
\newtheorem{coro}[prop]{Corollary}
\newtheorem{theo}{Theorem}[section]

\newtheorem{exam}[prop]{Example}
\newtheorem{rema}[prop]{Remark}

\newif\ifdraft
\drafttrue
\ifdraft
\else
\fi
\numberwithin{equation}{section}
\begin{document}

\title[Constructing Whitney sets via IFS with condensation]{Constructing Whitney sets via IFS with condensation}
\author{Qi-han Yuan} \address{Department of Mathematics and Statistics, Central China Normal University, Wuhan, 430079, China
} \email{yuanqihan\_ccnu@sina.com
 }

\author{Yuan Zhang$^*$} \address{Department of Mathematics and Statistics, Central China Normal University, Wuhan, 430079, China
} \email{yzhang@mail.ccnu.edu.cn
 }

\date{\today}
\thanks {The work is supported by NSFS Nos. 11971195, 12071171 and 11601172.}


\thanks{* The correspondence author.}

\begin{abstract}
In 1935, Whitney constructed a smooth function for which the Morse-Sard Theorem does not hold.
Whitney's  construction is closely related to certain compact connected set, which is called Whitney set now.
From then on, there are a lot of works on Whitney sets.
 In this paper, we use IFS with condensation, a notion   introduced by Barnsley and Demko in 1985, to construct Whitney arcs and Whitney sets.
 Our construction includes  most early results as special cases.
\end{abstract}
\maketitle

\section{Introduction}
In 1935, Whitney \cite{Whitney_1935} published an example of an arc $\gamma$ in $\mathbb{R}^{2}$ with the property that, there exists a $C^{1}$ function $f:\mathbb{R}^{2}\rightarrow\mathbb{R}$ such that the gradient $\nabla f|_{\gamma}\equiv 0$, but $f|_{\gamma}$ is not a constant function. Recall that any point at which all its first partial derivatives vanish is called a \emph{critical point} of $f$, and a connected component of the set of critical points is called a \emph{critical set} of $f$. In Whitney's example (see Figure 1 (b)), a critical set of $f$ contains an arc of Hausdorff dimension $\log4 / \log3$, and the image of this arc contains an interval. This shows that the Morse-Sard Theorem does not hold when the $C^{2}$-smoothness is replaced by $C^{1}$-smoothness. In general, the Whitney set is defined as follows.

\begin{defi}
A compact connected set $E\subset\mathbb{R}^{s}$ is said to be a \textbf{Whitney set}, if there is a $C^{1}$ function $f:\mathbb{R}^{s}\rightarrow\mathbb{R}$ such that the gradient $\nabla f|_{E}\equiv 0$ and $f|_{E}$ is not a constant.
\end{defi}

There are some works devoted to the construction of Whitney sets. Norton \cite{Norton_1989} showed that if $\gamma$ is a $t$-quasi-arc with $t < \dim_{H}\gamma$, then $\gamma$ is a Whitney set. Wen and Xi \cite{Xi_2003} proved that any self-similar arc of Hausdorff dimension$\ > 1$ is a Whitney set. (An arc $\gamma$ in $\mathbb{R}^{s}$ is said to be a \emph{$t$-quasi-arc} with $t\geq 1$, if there is a constant $C$ such that $\mid\gamma_{\widehat{xy}}\mid^{t} \leq C\mid x-y\mid$ for any $x,y\in \gamma$, where $\gamma_{\widehat{xy}}$ is the subarc connecting $x$ and $y$, and $\mid\gamma_{\widehat{xy}}\mid$ denotes the diameter.)

It is also shown that the following two kinds of sets are not Whitney sets: If every pair of points in $E$ can be connected by a rectifiable arc lying in $E$, then $E$ is not a Whitney set (Whyburn \cite{Whyburn_1929}, 1929); the graph $G$ of any continuous function $g:\mathbb{R}\rightarrow\mathbb{R}$ is not a Whitney set (Choquet \cite{Choquet_1944}, 1944).

Wen and Xi \cite{Xi_2009} gave a criterion for  Whitney set, according to certain nested decompositions of $E$. Especially, they showed that some Moran constructions are Whitney sets. However, the conditions in the criterion are not always easy to check. For other works on Whitney set, see \cite{Bates_1996,Besicovitch_1961,Norton_1986,Souto_2019}.

Let $\{\varphi_{i}\}_{i=1}^{N}$ be an iterated function system (IFS) on $\mathbb{R}^{s}$ such that
  all $\varphi_{i}$ are  contractive similitudes.
We shall denote  by  $\rho_i$ the contraction ratio of $\varphi_i$.
 The unique nonempty compact subset $E$ satisfying
 $
 E=\bigcup\limits_{i=1}^{N} \varphi_{i}(E),
 $
 is called a \emph{self-similar set}, and the unique $s>0$ satisfying $\Sigma_{i=1}^{N}\rho_{i}^{s}=1$ is called the \emph{similarity dimension} of $\{\varphi_{i}\}_{i=1}^{N}$. See Hutchinson \cite{Hutchinson_1981}.
 We also call $E$ the \emph{attractor} of $\{\varphi_{i}\}_{i=1}^{N}$.

In the present paper, we use IFS with condensation, a notion introduced by Barnsley and Demko \cite{Barnsley_1985}, to construct Whitney sets.
Let $\{\varphi_{i}\}_{i=1}^{N}$ be an IFS on $\mathbb{R}^{s}$,
 and $C$ be a compact subset of $\mathbb{R}^{s}$. Then  (\cite{Barnsley_1985})
there is a unique nonempty compact subset $K$ of $\mathbb{R}^{s}$ satisfying
$$K=C\cup\bigcup\limits_{i=1}^{N} \varphi_{i}(K).$$
Following \cite{Barnsley_1985}, we call $\mathcal{S} \! = \! \{C; \varphi_{1}, \varphi_{2}, \ldots, \varphi_{N}\}$ an \emph{IFS with condensation}, $C$ the \emph{condensation set},  and  $K$ the \emph{attractor} of $\mathcal{S}$.

For $\boldsymbol{\omega}=(\omega_{1}, \ldots, \omega_{k})\in \{1,2,\ldots, N\}^{k}$, $k\geq 0$, we denote by $|\boldsymbol{\omega}|:=k$ the length of $\boldsymbol{\omega}$.
We define $\varphi_{\boldsymbol{\omega}}=\varphi_{\omega_1 \ldots \omega_k}=\varphi_{\omega_1}\circ\ldots\circ \varphi_{\omega_k}$ and $\rho_{\boldsymbol{\omega}}=\prod_{i=1}^k \rho_{\omega_{i}} $. Set $\varphi_{\varepsilon}=id$ and $\rho_{\varepsilon}=1$ ,
where $\varepsilon$ is the empty-word.
By Ma and Zhang \cite{Ma_2020},
\begin{equation}\label{a}
K=\bigg(\bigcup\limits_{0\leq|\boldsymbol{\omega}|<k} \varphi_{\boldsymbol{\omega}}(C)\bigg)\cup\bigg(\bigcup\limits_{|\boldsymbol{\omega}|=k} \varphi_{\boldsymbol{\omega}}(K)\bigg),
\end{equation}
and
\begin{equation}\label{b}
K=\bigg(\bigcup\limits_{|\boldsymbol{\omega}|\geq 0} \varphi_{\boldsymbol{\omega}}(C)\bigg)\cup E,
\end{equation}
where $E$ is the attractor of the IFS $\{\varphi_{i}\}_{i=1}^{N}$.

We say that an arc $\gamma$ can be decomposed into $\gamma_{1}+\gamma_{2}+\cdots+\gamma_{m}$, if all $\gamma_{\ell}$ are subarcs of $\gamma$, and $\gamma$ is a joining of $\gamma_{1},\gamma_{2},\cdots,\gamma_{m}$. Our main results is the following.

\begin{theo}\label{Main_result_1}
Let $\mathcal{S}=\{C; \varphi_{1}, \varphi_{2}, \ldots, \varphi_{N}\}$ be an IFS with condensation, such that the condensation set $C=\bigcup\limits_{j=1}^{m} C_{j}$ is a disjoint union of arcs $C_{j}$ with $m\geq 2$. Let $K$ be the attractor of $\mathcal{S}$. Then $K$ is a Whitney set, if the following conditions hold:

$(\textrm{i})$ $K$ is an arc;

$(\textrm{ii})$ The arc $K$ can be decomposed into $\gamma_{1}+\gamma_{2}+\cdots+\gamma_{m+N}$, where $\gamma_{1},\gamma_{2},\ldots,\gamma_{m+N}$ is a rearrangement of $C_{1},\ldots,C_{m},\varphi_{1}(K),\ldots,\varphi_{N}(K)$ such that $\gamma_{1}, \gamma_{m+N}\in\{C_{1},\ldots,C_{m}\}$;

$(\textrm{iii})$ $\dim_{H}E>1$, where $E$ is the attractor of the IFS $\{\varphi_{i}\}_{i=1}^{N}$.
\end{theo}

An arc $\gamma$ is called a \emph{self-similar arc}, if $\gamma$ is the attractor of a family of contractive similitudes $\{\varphi_{i}\}_{i=1}^{N}$  such that
 $\varphi_{i}(\gamma)\cap \varphi_{j}(\gamma)$ is a singleton when $|i-j| = 1$
and
 $\varphi_{i}(\gamma)\cap \varphi_{j}(\gamma)=\emptyset$ when $|i-j| > 1$.

As an application of Theorem \ref{Main_result_1}, we show that:

\begin{coro}\label{Corollary_1}
If $\gamma$ is a self-similar arc with $\dim_{H}\gamma>1$, then $\gamma$ is a Whitney set.
\end{coro}

\begin{rema}
\rm This result was proved by Wen and Xi \cite{Xi_2003} under the open set condition. We note that Kamalutdinov, Tetenov and Vaulin \cite{Tetenov_2020} showed that there exists a self-similar arc in $\mathbb{R}^{3}$ which does not satisfy the open set condition.
\end{rema}

Next, we generalize Theorem \ref{Main_result_1} to a setting which requires $K$ has a chain structure instead of being an arc. First, we introduce a notion of linear IFS with condensation, which is a generalization of zipper introduced by Aseev, Tetenov and Kravchenko \cite{Tetenov_2003}.

Let $A$ be a compact set, and  let $a,b$ be two points in  $A$. We  shall call $(A,\{a,b\})$ a \emph{button} with two endpoints $a$ and $b$. If $(A',\{a',b'\})$ is a button such that $A'\subset A$, then we call $(A',\{a',b'\})$ a \emph{sub-button} of $(A,\{a,b\})$. We say
\begin{equation}\label{c}
(A,\{a,b\}) \! = \! (A_{1},\{a_{1},b_{1}\})+(A_{2},\{a_{2},b_{2}\})+\cdots+(A_{m},\{a_{m},b_{m}\})
\end{equation}
is a \emph{zipper decomposition} of $(A,\{a,b\})$, if the following conditions hold:

$(\textrm{i})$ $A=\bigcup_{j=1}^m A_j,$  and all $(A_{\ell},\{a_{\ell},b_{\ell}\})$ are sub-buttons of $(A,\{a,b\})$;

$(\textrm{ii})$ the endpoints satisfy $a \! = \! a_1, b_\ell \! = \! a_{\ell+1} (1\leq \ell \leq m-1),  b_{m} \! = \! b$;

$(\textrm{iii})$ for arbitrary $\ell',\ell''\in\{1,\ldots,m\}$, we have $\#(A_{\ell'}\cap A_{\ell''}) \! = \! \begin{cases} 0,& \text{if} \ \ |\ell'-\ell''|> 1 \\ 1,& \text{if} \ \ |\ell'-\ell''|=1. \end{cases}$

Let $\mathcal{S}=\{C; \varphi_{1}, \varphi_{2}, \ldots, \varphi_{N}\}$ be an IFS with condensation. From now on, we assume that the condensation set $C$ is a finite disjoint union of compact connected sets, that is, $C=\bigcup_{j=1}^{m} C_{j}$, $m\geq2$, $C_{j}$ are connected components of $C$ and $C_{j}$ are compact; in this case, we call $\mathcal{S}$ an \emph{IFS with finite-component condensation}.

\begin{defi}{\rm Let $\mathcal{S}=\{C; \varphi_{1}, \varphi_{2}, \ldots, \varphi_{N}\}$ be an IFS with finite-component condensation, and let $K$ be its attractor.
If there exist $\{a,b\} \! \subseteq \! K$ and $\{a_{j},b_{j}\} \! \subseteq \! C_{j}$, $1\leq j\leq m$, such that
\begin{equation}\label{d}
(K,\{a,b\})=A_{1}+A_{2}+\cdots+A_{m+N}
\end{equation}
is a zipper decomposition of $(K,\{a,b\})$, where $A_{1},\ldots,A_{m+N}$ is a rearrangement of \begin{equation}\label{eq:Aj}
\big(C_{1},\{a_{1},b_{1}\}\big) \ , \ldots, \ \big(C_{m},\{a_{m},b_{m}\}\big) \ , \ (\varphi_{1}(K),\{\varphi_{1}(a),\varphi_{1}(b)\}) \ ,  \ldots, \ (\varphi_{N}(K),\{\varphi_{N}(a),\varphi_{N}(b)\}),
\end{equation}
then we call
 \begin{equation}\label{eq:linear}
 \tilde {\mathcal {S}}=\{(C_{1},\{a_{1},b_{1}\}), \ldots,(C_{m},\{a_{m},b_{m}\}); \varphi_{1}, \varphi_{2}, \ldots, \varphi_{N}\}
 \end{equation}
a \emph{linear IFS with condensation}, and call $(K,\{a,b\})$ the \emph{attractor} of $\tilde {\mathcal{S}}$.
}\end{defi}

For an illustration, see Figure 1(a).

\begin{theo}\label{Main_result_2}
Let $(K,\{a,b\})$ be the attractor of a linear IFS with condensation defined in \eqref{eq:linear}. Then $K$ is a Whitney set provided that

$(\textrm{i})$ $A_{1}, A_{m+N}\in \{(C_{1},\{a_{1},b_{1}\}),\ldots ,(C_{m},\{a_{m},b_{m}\})\}$, where $A_{1}, A_{m+N}$ are defined in \eqref{d} and \eqref{eq:Aj};

$(\textrm{ii})$ $\dim_{H}E>1$, where $E$ is the attractor of the IFS $\{\varphi_{i}\}_{i=1}^{N}$.
\end{theo}

\begin{figure}[h]
\centering
    \begin{tikzpicture}[font=\tiny,scale=.84, declare function={xx(\t)=(2*\t*\t-10*\t+9)*(\t-2)*(\t-3)*.25*\l;
        yy(\t)=(6*\t*\t-44*\t+81)*(\t-1)*(\t-2)/6*\l/2;}]
        \pgfmathsetmacro{\r}{1.8}
        \pgfmathsetmacro{\l}{7.8}
        \pgfmathsetmacro{\h}{\l+1.2}
        \pgfmathsetmacro{\y}{.65}
        \coordinate (q1) at (1,.75);
        \coordinate (q2) at (.75,3.5);
        \coordinate (q3) at (5,.75);
        \coordinate (q4) at (5,4.75);
        \coordinate (qb1) at ($(q1)+(\r,\r)$);
        \coordinate (qb2) at ($(q2)+(2*\r,2*\r)$);
        \coordinate (qb3) at ($(q3)+(\r,\r)$);
        \coordinate (qb4) at ($(q4)+(\r,\r)$);

        \draw(\l,.5*\l)node[right]{$Q$};
        \draw(0,0)node[below left]{$a$}rectangle++(\l,\l) node[above right]{$b$};
        \foreach \i in {{0,0},{\l,\l}} \draw[fill=black](\i)circle(.05);
        \foreach \i/\j/\tex/\sty/\styl in {
        q1/\r/1/above left/right,
        q2/2*\r/2/below/above,
        q3/\r/3/below/above right,
        q4/\r/4/above right/above left}
            {\draw(\i)node[\sty]{$\varphi_\tex(a)$}rectangle++(\j,\j) node[\styl]{$\varphi_\tex(b)$};
             \draw($(\i)+(.5*\j,.5*\j)$)node{$Q_\tex$};
             \foreach \a in {{0,0},{\j,\j}}
                \draw[fill=black]($(\i)+(\a)$)circle(.05);
            }
        \foreach \i/\j/\k/\tex in {{0,0}/q1/{-.2,.1}/1, qb1/q2/{.1,.2}/2,qb2/q3/{.2,.2}/3,qb3/q4/{.2,.3}/4, qb4/{\l,\l}/{-.15,.15}/5}
            {\path[fill=red!50](\i)--($.5*(\i)+.5*(\j)+(\k)$)--(\j)-- ($.5*(\i)+.5*(\j)-(\k)$)--cycle;
            \draw($.5*(\i)+.5*(\j)$)node{$C_\tex$};}

        \draw(\h,0)rectangle++(\l,\l);
        \foreach \i/\sty/\tex in {{\l*.5,0}/below/a, {\l,\l*.5}/right/b}
            \draw($(\i)+(\h,0)$)node[\sty]{$\tex$};
        \foreach \i in {1,2,3,4}
            {\draw({xx(\i)+\y+\h},{yy(\i)+\y})rectangle++(\l/3,\l/3);
            \draw({xx(\i)+\y+\h+\l/6},{yy(\i)+\y+\l/6})node{$Q_\i$};}
        \foreach \i/\j/\loc in {1/2,4/4}
            \draw[red,fill=black]({\h+xx(\i)+\y},{yy(\i)+\l/6+\y})circle(.05)-- node[font=\tiny,above,text=black]{$C_\j$}++(-\l/6,0)circle(.05);
        \draw[red,fill=black]({\h+xx(2)+\y+\l/6},{yy(2)+\l/3+\y}) circle[radius=.05,black]--node[font=\tiny,right,text=black] {$C_3$}++(0,\l/6)circle[radius=.05,black];
        \foreach \i/\j/\loc/\sty/\styl in {1/1/{\h+\l/2,0}/above/black, 4/5/{\h+\l,\l/2}/below/near end}
            \draw[red,fill=black]({\h+xx(\i)+\y+\l/6},{yy(\i)+\y}) circle[radius=.05,black]--node[\sty,\styl,font=\tiny,text=black] {$C_\j$}(\loc)circle[radius=.05,black];
        \foreach \i/\tex in {0/(a) A variation of Whitney's example, \h/(b) Whitney's example}
        \draw(\i+\l/2,-1.2*\y)node[font=\normalsize]{\tex};
    \end{tikzpicture}
\caption{}
\end{figure}


We call $\varphi:\mathbb{R}^{2}\rightarrow\mathbb{R}^{2}$ a homotopy if $\varphi(x,y)=\rho(x,y)+(a_{1},a_{2})$ for some constant $\rho>0$ and vector $(a_{1},a_{2})$.

\begin{exam}\rm Let $Q=[0,1]^{2}$ be the unit square. Let $\{\varphi_{i}\}_{i=1}^{4}$ be a family of contractive homotopies such that the similarity dimension is greater than 1, and $\varphi_{i}(Q)=Q_i$ are disjoint squares in $Q$ (see Figure 1 (a)). Let $C_{j},1\leq j\leq 5$, be the rhombuses indicated by Figure 1 (a).
Denote by $K$ the attractor of the IFS with finite-component condensation $\mathcal{S} \! = \! \{\bigcup_{j=1}^{5}C_{j}; \varphi_{1},\ldots,  \varphi_{4}\}$.

Let $a \! = \! (0,0),b \! = \! (1,1)$. It is easy to prove
\begin{align}\label{e}
(K,\{a,b\})= \ &(C_{1},\{a,\varphi_{1}(a)\})+(\varphi_{1}(K),\{\varphi_{1}(a),\varphi_{1}(b)\}) \ + \notag\\ &\sum_{j=2}^{4}\left((C_{j},\{\varphi_{j-1}(b),\varphi_{j}(a)\}) +(\varphi_{j}(K),\{\varphi_{j}(a),\varphi_{j}(b)\})\right)+ (C_{5},\{\varphi_{4}(b),b\})
\end{align}
is a zipper decomposition of $(K,\{a,b\})$(see Figure 1 (a)). Therefore, $(K,\{a,b\})$ is the attractor of the linear IFS with condensation $\tilde {\mathcal {S}} \! = \! \{(C_{1},\{a,\varphi_{1}(a)\}), \ldots , (C_{5},\{\varphi_{4}(b),b\}); \varphi_{1}, \ldots,\varphi_{4}\}$.
Hence, $K$ is a Whitney set by Theorem \ref{Main_result_2}.
\end{exam}

\begin{rema}
\rm In the above example, if these rhombuses $C_{j}$ degenerate to line segments,      then
 we obtain Whitney's construction in \cite{Whitney_1935}.
\end{rema}

\section{Proof of Theorem \ref{Main_result_1}}

The following lemma is a simple version of the Whitney Extension Theorem.

\begin{lem}(\cite{Whitney_1934})\label{lem-extention} Suppose $E\subset\mathbb{R}^{s}$ is compact and $f:E\rightarrow\mathbb{R}$ is a real function. If for any $\epsilon> 0$, there exists $\delta> 0$ such that for any $x,y\in E$ with $0<\mid x-y \mid<\delta$, $$\frac{\mid f(x)-f(y) \mid}{\mid x-y \mid}<\epsilon,$$ then there exists a $C^{1}$ extension $F:\mathbb{R}^{s}\rightarrow\mathbb{R}$ such that $F|_{E}=f$ and $\nabla F|_{E}\equiv 0$.
\end{lem}

Let $A,B\subset\mathbb{R}^{s}$, we denote the \emph{distance} between $A$ and $B$ by $$dist(A,B)=\inf \ \{ \ \mid x-y \mid \ : x\in A , y\in B\}.$$

\begin{proof}[\textbf{Proof of Theorem \ref{Main_result_1}}]\ Let $a,b$ be the endpoints of the arc $K$, which we will denote by $\Gamma_{\widehat{ab}}$.
For $z_{1},z_{2}\in\Gamma_{\widehat{ab}}$, we denote by $\Gamma_{\widehat{z_{1}z_{2}}}$ the subarc of $\Gamma_{\widehat{ab}}$ between $z_{1}$ and $z_{2}$. Denote $s=\dim_{H}E$. We define a function $f:\Gamma_{\widehat{ab}}\longrightarrow \mathbb{R}$ by $$f(z)=\mathcal{H}^{s}(\Gamma_{\widehat{az}}\cap E),\ z\in\Gamma_{\widehat{ab}}.$$

To apply Lemma \ref{lem-extention}, we need to estimate
$$\frac{\mid f(z_{1})-f(z_{2}) \mid}{\mid z_{1}-z_{2} \mid}=\frac{\mathcal{H}^{s}(\Gamma_{\widehat{z_{1}z_{2}}}\cap E)}{\mid z_{1}-z_{2} \mid} \quad \text{for any}~ z_{1},z_{2}\in \Gamma_{\widehat{ab}}.$$

Pick $z_{1},z_{2}\in \Gamma_{\widehat{ab}}$ with $z_{1}\neq z_{2}$. Let $\boldsymbol{\nu}=\nu_{1}\nu_{2}\cdots\nu_{k_{0}}\in \{1,2,\ldots, N\}^{k_{0}}$ be the longest word such that $\varphi_{\boldsymbol{\nu}}(\Gamma_{\widehat{ab}})$ contain both $z_{1}$ and $z_{2}$ (let $\boldsymbol{\nu}=\varepsilon$ and
$\varphi_{\boldsymbol{\nu}}=id$
if $z_1\in \gamma_{\ell_1}$ and $z_2\in \gamma_{\ell_2}$ with $\ell_1\neq \ell_2\in\{1,\dots,m+N\}$
or $\gamma_{\ell_1}=\gamma_{\ell_2}\in\{C_1,\dots,C_m\}$).
Denote by $\rho_{\boldsymbol{\nu}}$ the contraction ratio of $\varphi_{\boldsymbol{\nu}}$.
Since $\Gamma_{\widehat{ab}}$ can be decomposed into
\begin{equation}\label{f}
\gamma_{1} + \gamma_{2} + \cdots + \gamma_{m+N},
\end{equation}
where $\gamma_{1},\gamma_{2},\ldots,\gamma_{m+N}$ is a rearrangement of $C_{1},\ldots,C_{m},\varphi_{1}(\Gamma_{\widehat{ab}}),\ldots,\varphi_{N}(\Gamma_{\widehat{ab}})$, we have $z_{1} \! \in \! \varphi_{\boldsymbol{\nu}}(\gamma_{\ell_{1}})$ and $z_{2} \! \in \! \varphi_{\boldsymbol{\nu}}(\gamma_{\ell_{2}})$ for some $\ell_{1}, \ell_{2} \! \in \! \{1,\ldots,m+N\}$.

Case 1: $\gamma_{\ell_{1}}=\gamma_{\ell_{2}}$.

$\gamma_{\ell_{1}}=\gamma_{\ell_{2}}$ can not both fall into the set $\{\varphi_{1}(\Gamma_{\widehat{ab}}),\ldots,\varphi_{N}(\Gamma_{\widehat{ab}})\}$ by the choice of $\boldsymbol{\nu}$, thus $\gamma_{\ell_{1}} \! = \! \gamma_{\ell_{2}} \! \in \! \{C_{1},\ldots,C_{m}\}$, hence
\begin{equation}\label{g}
\mathcal{H}^{s}(\Gamma_{\widehat{z_{1}z_{2}}}\cap E) \! = \! 0.
\end{equation}

Case 2: $\gamma_{\ell_{1}}\neq\gamma_{\ell_{2}}\in\{C_{1},\ldots,C_{m}\}$.

Set $\xi_{1}:=\underset{j'\neq j''}{\min} \{dist(C_{j'},C_{j''})\}$,
we have $\xi_{1}>0$ for $C_j, 1\leq j \leq m$ are pairwise disjoint.
Hence,
\begin{equation}\label{h}
\frac{\mid f(z_{1})-f(z_{2}) \mid}{\mid z_{1}-z_{2} \mid}\leq\frac{\mathcal{H}^{s}(\varphi_{\boldsymbol{\nu}}(\Gamma_{\widehat{ab}})\cap E)}{\rho_{\boldsymbol{\nu}} \ \xi_{1}}=\frac{\mathcal{H}^{s}(\varphi_{\boldsymbol{\nu}}(E))}{\rho_{\boldsymbol{\nu}} \ \xi_{1}}= \frac{(\rho_{\boldsymbol{\nu}})^{s} \ \mathcal{H}^{s}(E)}{\rho_{\boldsymbol{\nu}} \ \xi_{1}}= M_{1} (\rho_{\boldsymbol{\nu}})^{s-1},
\end{equation}
where the first inequality holds by $\Gamma_{\widehat{z_{1}z_{2}}}\subseteq\varphi_{\boldsymbol{\nu}}(\Gamma_{\widehat{ab}})$, and $M_{1}=\frac{\mathcal{H}^{s}(E)}{\xi_{1}}$.

Case 3: $\gamma_{\ell_{1}}\neq\gamma_{\ell_{2}}$ and at least one of them belongs to $\{\varphi_{1}(\Gamma_{\widehat{ab}}),\ldots,\varphi_{N}(\Gamma_{\widehat{ab}})\}$.

Without loss of generality, we assume that $\gamma_{\ell_{1}}=\varphi_{u}(\Gamma_{\widehat{ab}}), u\in\{1,2,\ldots, N\}$.
Since $\Gamma_{\widehat{ab}}$ can be decomposed into \eqref{f}, we have $z_{1}\in\varphi_{\boldsymbol{\nu}u}(\gamma_{p_{1}})$ for some $p_{1}\in\{1,\ldots,m+N\}$.

First, we consider the case $\ell_{2}>\ell_{1}$.

Case 3.1: If $\ell_{2}-\ell_{1}>1$ or $\gamma_{p_{1}}\neq \gamma_{m+N}$, then $\varphi_{\boldsymbol{\nu}u}(\gamma_{p_{1}}) \bigcap \varphi_{\boldsymbol{\nu}}(\gamma_{\ell_{2}})=\emptyset$.
Set $$\xi_{2}=\min\limits_{\substack{\ell',\ell''\in\{1,\ldots,m+N\}\\  u\in\{1,2,\ldots,N\}}}\{dist\big(\varphi_{u}(\gamma_{\ell'}),\gamma_{\ell''}\big); \varphi_{u}(\gamma_{\ell'})\bigcap\gamma_{\ell''} )=\emptyset\},$$
 we have $\xi_2>0$ by finite choices property.
Then \begin{equation}\label{i}\frac{\mid f(z_{1})-f(z_{2}) \mid}{\mid z_{1}-z_{2} \mid}\leq\frac{\mathcal{H}^{s}(\varphi_{\boldsymbol{\nu}}(\Gamma_{\widehat{ab}})\cap E)}{\rho_{\boldsymbol{\nu}} \ \xi_{2}}=\frac{\mathcal{H}^{s}(\varphi_{\boldsymbol{\nu}}(E))}{\rho_{\boldsymbol{\nu}} \ \xi_{2}}= \frac{(\rho_{\boldsymbol{\nu}})^{s} \ \mathcal{H}^{s}(E)}{\rho_{\boldsymbol{\nu}} \ \xi_{2}}= M_{2} (\rho_{\boldsymbol{\nu}})^{s-1},\end{equation} where $M_{2}=\frac{\mathcal{H}^{s}(E)}{\xi_{2}}$.

Case 3.2: If $\ell_{2}-\ell_{1}=1$ and $\gamma_{p_{1}}=\gamma_{m+N}$, then the arc $\varphi_{\boldsymbol{\nu}u}(\gamma_{m+N})$ containing $z_1$ is adjacent to the arc $\varphi_{\boldsymbol{\nu}}(\gamma_{\ell_{2}})$ containing $z_2$.
If $z_2\in\varphi_{\boldsymbol{\nu}}(\gamma_{\ell_{2}})=\varphi_{\boldsymbol{\nu}}(C_j)$ for some $j\in\{1,\dots,m\}$ or
$z_2\in \varphi_{\boldsymbol{\nu}u'}(\gamma_1)$ for some $u'\in \{1,\dots,N\}$, then
\begin{equation}\label{k}
\mathcal{H}^{s}(\Gamma_{\widehat{z_{1}z_{2}}}\cap E)=0.
\end{equation}
Otherwise, $z_2\in \varphi_{\boldsymbol{\nu}u'}(\gamma_{p_2})$ for some $u'\in \{1,\dots,N\}$ and $p_2\in\{2,\dots,m+N\},$
 then $\varphi_{\boldsymbol{\nu}u}(\gamma_{m+N}) \bigcap \varphi_{\boldsymbol{\nu}u'}(\gamma_{p_{2}})\\ =\emptyset$. Set $$\xi_{3}=\min\limits_{\substack{\ell\in\{2,\ldots,m+N\}\\  u,u'\in\{1,2,\ldots,N\}}}\{dist\big(\varphi_{u}(\gamma_{m+N}),\varphi_{u'}(\gamma_{\ell})\big); \varphi_{u}(\gamma_{m+N})\cap\varphi_{u'}(\gamma_{\ell})=\emptyset\},$$
  we have $\xi_{3}>0$. Then
  \begin{equation}\label{l}\frac{\mid f(z_{1})-f(z_{2}) \mid}{\mid z_{1}-z_{2} \mid}\leq\frac{\mathcal{H}^{s}(\varphi_{\boldsymbol{\nu}}(\Gamma_{\widehat{ab}})\cap E)}{\rho_{\boldsymbol{\nu}} \ \xi_{3}}=\frac{\mathcal{H}^{s}(\varphi_{\boldsymbol{\nu}}(E))}{\rho_{\boldsymbol{\nu}} \ \xi_{3}}= \frac{(\rho_{\boldsymbol{\nu}})^{s} \ \mathcal{H}^{s}(E)}{\rho_{\boldsymbol{\nu}} \ \xi_{3}}= M_{3} (\rho_{\boldsymbol{\nu}})^{s-1},\end{equation} where $M_{3}=\frac{\mathcal{H}^{s}(E)}{\xi_{3}}$.

By symmetry, the case $\ell_{2} \! < \! \ell_{1}$ is similar as above.

Let $|z_{1}-z_{2}| \! \rightarrow \! 0$, then $|\boldsymbol{\nu}| \! \rightarrow \! \infty$ which lead to $\rho_{\boldsymbol{\nu}} \! \rightarrow \! 0$, it follows that $(\rho_{\boldsymbol{\nu}})^{s-1} \! \rightarrow \! 0$ by $s \! > \! 1$. Summing up with \eqref{g}-\eqref{l}, $| f(z_{1})-f(z_{2})| \! = \! o(| z_{1}-z_{2}|)$ holds for any $z_{1},z_{2}\in\Gamma_{\widehat{ab}}$, then $\Gamma_{\widehat{ab}}$ is a Whitney set by Lemma \ref{lem-extention}.
The theorem is proved.
\end{proof}

\section{Proof of Corollary \ref{Corollary_1}}

\begin{proof}[\textbf{Proof of Corollary \ref{Corollary_1}}]
Suppose that $\gamma$ is a self-similar arc generated by a family of contractive similitudes $\{S_{i}\}_{i=1}^{n}$
with contraction ratios $\{\rho_{i}\}_{i=1}^{n}$.
Then $\gamma$ can be decomposed into subarcs as \begin{equation}\label{m}
\gamma=\bigcup_{\boldsymbol{\omega}\in\{1,\ldots,n\}^{k}}S_{\boldsymbol{\omega}}(\gamma)=\gamma_{1}+\gamma_{2}+\cdots+\gamma_{n^{k}},\quad k\geq 0.
\end{equation}

Next, we construct an IFS with condensation $\mathcal{S}$ such that $\gamma$ is the attractor of it
and satisfies Theorem \ref{Main_result_1}.

Fix $k \! \geq \! 1$. For any $\ell \! \in \! \{1,2,\ldots,n^{k}\}$, denote by $\boldsymbol{\omega}_{(\ell)}$ the unique word in $\{1,\ldots,n\}^{k}$ such that $\gamma_{\ell} \! = \! S_{\boldsymbol{\omega}_{(\ell)}}(\gamma)$. Then $\gamma$ is the attractor of the IFS with condensation $\mathcal{S}_{k} \! = \! \{\gamma_{1}\bigcup \gamma_{n^{k}}; S_{\boldsymbol{\omega}_{(2)}}, S_{\boldsymbol{\omega}_{(3)}}, \ldots, S_{\boldsymbol{\omega}_{(n^{k}-1)}}\}$ satisfying Theorem \ref{Main_result_1} $(\textrm{i})$ $(\textrm{ii})$.

Let $E_{k}$ be the attractor of the IFS $\{S_{\boldsymbol{\omega}_{(\ell)}}\}_{\ell=2}^{n^{k}-1}$. It will suffice to prove $\dim_{H}E_{k}> 1$ for some integer $k>0$. Denote by $s$ the similarity dimension of $\{S_{i}\}_{i=1}^{n}$, that is $\sum_{i=1}^{n}\rho_{i}^{s}=1$, then we have $\sum_{\ell=1}^{n^{k}}\rho_{\boldsymbol{\omega}_{(\ell)}}^{s}=1$. Since $\dim_{H}\gamma > 1$, then $s=\dim_{s}\gamma\geq\dim_{H}\gamma > 1$. Clearly, there is a constant $N_{1}>0$, such that $\rho_{\boldsymbol{\omega}_{(1)}}^{s}+\rho_{\boldsymbol{\omega}_{(n^{k})}}^{s}<\frac{1}{2}$ when $k\geq N_{1}$.

Denote $\rho^{\ast}=\max_{1\leq i\leq n}\{\rho_{i}\}$. Let $t'\in(1,s)$. For $k\geq N_{1}$, we have $$\sum_{\ell=2}^{n^{k}-1}\rho_{\boldsymbol{\omega}_{(\ell)}}^{t'}\geq\frac{\sum_{\ell=2}^{n^{k}-1}\rho_{\boldsymbol{\omega}_{(\ell)}}^{s}}{(\rho^{\ast})^{k(s-t')}} \geq\frac{1-\frac{1}{2}}{(\rho^{\ast})^{k(s-t')}}.$$
Since $(\rho^{\ast})^{k(s-t')}\rightarrow 0$ as $k\rightarrow \infty$, we have $\sum_{\ell=2}^{n^{k}-1}\rho_{\boldsymbol{\omega}_{(\ell)}}^{t'}>1$ for $k$ large enough. Thus there exist $t\in(t',s)$, such that $\sum_{\ell=2}^{n^{k}-1}\rho_{\boldsymbol{\omega}_{(\ell)}}^{t}=1$ for $k$ large enough, then $\dim_{s}E_{k}=t>t'>1$. Since $E_{k}$ satisfies the strong separation condition, we have $\dim_{H}E_{k}=\dim_{s}E_{k}>1$ for $k$ large enough.

Finally, $\gamma$ is a Whitney set by Theorem \ref{Main_result_1}.
\end{proof}

By the arbitrary of $t'<s$ in the above proof, we have $\dim_{s}\gamma\leq\dim_{H}\gamma$ for arbitrary self-similar arc $\gamma$. Then we have the following corollary even if $\gamma$ does not satisfy the open set condition.

\begin{coro}
If $\gamma$ is a self-similar arc, then $\dim_{s}\gamma=\dim_{H}\gamma$.
\end{coro}

\section{Proof of Theorem \ref{Main_result_2}}

Suppose a button $(A,\{a,b\})$ has a zipper  decomposition as (\ref{c}),
that is $(A,\{a,b\}) \! = \! (A_{1},\{a_{1},b_{1}\})+(A_{2},\{a_{2},b_{2}\})+\cdots+(A_{m},\{a_{m},b_{m}\})$.
For any $\ell\in\{1,2,\ldots,m\}$, we define
\begin{equation}\label{eq:tail}
(\widetilde{A}_{b_{\ell}},\{a,b_{\ell}\})=\left (\bigcup_{j=1}^\ell A_{k},\{a  , b_{\ell}\}\right).
\end{equation}
For $1\leq \ell \leq m-1$,  we have
$(\widetilde{A}_{a_{\ell+1}},\{a,a_{\ell+1}\})=(\widetilde{A}_{b_{\ell}},\{a,b_{\ell}\})$
by $a_{\ell+1}=b_\ell$.
At the same time, we set $(\widetilde{A}_{a_1},\{a,a_1\})=\emptyset$.

In this section, we always assume that $(K,\{a,b\})$ is the attractor of a linear IFS with condensation $\tilde {\mathcal{S}} \! = \! \{(C_{1},\{a_{1},b_{1}\}), \ldots ,(C_{m},\{a_{m},b_{m}\}); \varphi_{1}, \varphi_{2}, \ldots, \varphi_{N}\}$. For $k\geq1$, we call 
$$\bigcup_{j=1}^m\{\varphi_{\boldsymbol{\omega}}(a_j),\varphi_{\boldsymbol{\omega}}(b_j)\}_{|\boldsymbol{\omega}|=k-1}, \\ \{\varphi_{\boldsymbol{\omega}}(a),\varphi_{\boldsymbol{\omega}}(b)\}_{|\boldsymbol{\omega}|=k}$$
the \emph{k-level node} of $(K,\{a,b\})$, where $\varphi_{\boldsymbol{\omega}}=id$ if $|\boldsymbol{\omega}|=0$.

To characterize the connectedness of the attractor $K$,
we define the \emph{Hata graph} of $ \tilde {\mathcal{S}}$ according to \cite{Ma_2020}.
Let $\{\varphi_{i}(K): 1\leq i\leq N\}\cup\{C_{1},\ldots,C_{m}\}$ be the vertex set; for its any two vertices
$U$ and $V$, there is an edge if and only if $U\cap V\neq\emptyset$. This Hata graph is said to be connected if for any two vertices, there is a path in the graph connecting them.
Ma and Zhang \cite{Ma_2020} proved that $K$ is connected if the Hata graph of $\tilde {\mathcal{S}}$ is connected.

\begin{lem}\label{lem-connect} Let $(K,\{a,b\})$ be the attractor of a linear IFS with condensation $\tilde {\mathcal{S}}$ defined in \eqref{eq:linear}. Then $K$ is connected.
\end{lem}

\begin{proof}
\rm Since $(K,\{a,b\})$ has a zipper decomposition as \eqref{d} and satisfy its condition (ii),
the Hata graph of $\tilde {\mathcal{S}}$ is connected,
so $K$ is connected by \cite{Ma_2020}.
\end{proof}

\begin{proof}[\textbf{Proof of Theorem \ref{Main_result_2}}]
By Lemma \ref{lem-connect}, $K$ is a compact connected set. Let $s=\dim_{H}E$. To apply Lemma \ref{lem-extention}, we firstly construct a function $f:(K,\{a,b\})\longrightarrow \mathbb{R}$ as following: let $z \! \in \! (K,\{a,b\})$, if $z$ is a node of $(K,\{a,b\})$, we set \begin{equation}\label{n}f(z) \! = \! \mathcal{H}^{s}\left((\widetilde{K}_{z},\{a,z\})\cap E\right);\end{equation} if $z\in(\varphi_{\boldsymbol{\omega}}(C_{j}),\{\varphi_{\boldsymbol{\omega}}(a_{j}), \varphi_{\boldsymbol{\omega}}(b_{j})\})$ for some $|\boldsymbol{\omega}|\geq 0$ and $j\in\{1,\ldots,m\}$, we set
\begin{equation}
\label{o}f(z)=\mathcal{H}^{s}\left((\widetilde{K}_{\varphi_{\boldsymbol{\omega}}(a_{j})},\{a,\varphi_{\boldsymbol{\omega}}(a_{j})\})\cap E\right);
\end{equation}
otherwise, by \eqref{b} there exist a sequence of nodes $\{z_{k}\}_{k=1}^\infty$ such that $z_{k}\rightarrow z$ as $k\rightarrow\infty$,
where $z_{k}$ is a $k$-level node, we set
\begin{equation}
\label{p}f(z)=\lim\limits_{k\rightarrow\infty}\mathcal{H}^{s}\left((\widetilde{K}_{z_{k}},\{a,z_{k}\})\cap E\right).
\end{equation}

Next, we estimate $\frac{\mid f(z_{1})-f(z_{2}) \mid}{\mid z_{1}-z_{2} \mid}$ for any $z_{1},z_{2}\in (K,\{a,b\})$, where the argument is very similar to the proof of Theorem \ref{Main_result_1}.

Pick $z_{1}\neq z_{2}\in (K,\{a,b\})$. Let $\boldsymbol{\nu}=\nu_{1}\nu_{2}\cdots\nu_{k_{0}}\in \{1,2,\ldots, N\}^{k_{0}}$ be the longest word such that $\varphi_{\boldsymbol{\nu}}(K)$ contain both $z_{1}$ and $z_{2}$. Denote by $\rho_{\boldsymbol{\nu}}$ the contraction ratio of $\varphi_{\boldsymbol{\nu}}$.
Since $(K,\{a,b\})$ has a zipper decomposition as \eqref{d}, we have $z_{1} \! \in \! \varphi_{\boldsymbol{\nu}}(A_{\ell_{1}})$ and $z_{2} \! \in \! \varphi_{\boldsymbol{\nu}}(A_{\ell_{2}})$ for some $\ell_{1}, \ell_{2} \! \in \! \{1,\ldots,m+N\}$.

Case 1: $A_{\ell_{1}}=A_{\ell_{2}}$.

$A_{\ell_{1}}=A_{\ell_{2}}$ can not both fall into the set $\{\varphi_i(K);1\leq i\leq N\}$, thus $A_{\ell_{1}}=A_{\ell_{2}}\in\bigcup_{i=1}^m\{\big(C_i,\{a_i,b_i\}\big)\}$ by the choice of $\boldsymbol{\nu}$, by \eqref{o}, we have \begin{equation}\label{q} \mid f(z_{1})-f(z_{2}) \mid=0.\end{equation}

Case 2: $A_{\ell_{1}}\neq A_{\ell_{2}}\in\bigcup_{i=1}^m\{\big(C_i,\{a_i,b_i\}\big)\}$.

By \eqref{o}, we have \begin{equation}\label{r}
\frac{\mid f(z_{1})-f(z_{2}) \mid}{\mid z_{1}-z_{2} \mid}
\leq\frac{\mathcal{H}^{s}(\varphi_{\boldsymbol{\nu}}(K)\bigcap E)}{\rho_{\boldsymbol{\nu}} \ \xi_{1}}
=\frac{\mathcal{H}^{s}(\varphi_{\boldsymbol{\nu}}(E))}{\rho_{\boldsymbol{\nu}} \ \xi_{1}}
= M_{1} (\rho_{\boldsymbol{\nu}})^{s-1},
\end{equation}
where $\xi_{1}=\underset{j'\neq j''}{\min} \{dist(C_{j'},C_{j''})\}$ and $M_{1}=\frac{\mathcal{H}^{s}(E)}{\xi_{1}}$.

Case 3: $A_{\ell_{1}}\neq A_{\ell_{2}}$ and at least one of them belongs to $\{(\varphi_{i}(K),\{\varphi_{i}(a),\varphi_{i}(b)\});1\leq i\leq N\}$.

Without loss of generality, we assume that $A_{\ell_{1}}=(\varphi_{u}(K),\{\varphi_{u}(a),\varphi_{u}(b)\}), ~u\in\{1,2,\ldots, N\}$. Since $(K,\{a,b\})$ has a zipper decomposition as \eqref{d}, we have $z_{1}\in\varphi_{\boldsymbol{\nu}u}(A_{p_{1}})$ for some $p_{1}\in\{1,\ldots,m+N\}$.

First, we consider the case $\ell_{2}>\ell_{1}$.

Case 3.1: If $\ell_{2}-\ell_{1}>1$ or $A_{p_1}\neq A_{m+N}$, by the definition of $f$ and $z_{1},z_{2}\in\varphi_{\boldsymbol{\nu}}(K)$, we have \begin{equation}\label{s}
\frac{\mid f(z_{1})-f(z_{2}) \mid}{\mid z_{1}-z_{2} \mid}
\leq\frac{\mathcal{H}^{s}(\varphi_{\boldsymbol{\nu}}(K)\bigcap E)}{\rho_{\boldsymbol{\nu}} \ \xi_2}
=\frac{\mathcal{H}^{s}(\varphi_{\boldsymbol{\nu}}(E))}{\rho_{\boldsymbol{\nu}} \ \xi_2}= M_2 (\rho_{\boldsymbol{\nu}})^{s-1},
\end{equation}
where $\xi_{2}=\min\limits_{\substack{\ell',\ell''\in\{1,\ldots,m+N\}\\  u\in\{1,2,\ldots,N\}}}\{dist(\varphi_{u}(A_{\ell'}), A_{\ell''}\big); \varphi_{u}(A_{\ell'})\bigcap A_{\ell''}=\emptyset\}$ and $M_{2}=\frac{\mathcal{H}^{s}(E)}{\xi_{2}}$.

Case 3.2: $\ell_{2}-\ell_{1}=1$ and $A_{p_1}=A_{m+N}$, then the button $\varphi_{\boldsymbol{\nu}u}(A_{m+N})$ containing $z_1$
is adjacent to the button $\varphi_{\boldsymbol{\nu}}(A_{\ell_{2}})$ containing $z_2$.
If $z_2\in\varphi_{\boldsymbol{\nu}}(A_{\ell_{2}})=\varphi_{\boldsymbol{\nu}}(C_j)$ for some $j\in\{1,\dots,m\}$ or
$z_2\in \varphi_{\boldsymbol{\nu}u'}(A_1)$ for some $u'\in \{1,\dots,N\}$, then
\begin{equation}\label{t}
\mid f(z_{1})-f(z_{2}) \mid=0.
\end{equation}
Otherwise, $z_2\in \varphi_{\boldsymbol{\nu}u'}(A_{p_2})$ for some $u'\in \{1,\dots,N\}$ and $p_2\in\{2,\dots,m+N\},$
 then $\varphi_{\boldsymbol{\nu}u}(A_{m+N}) \bigcap \varphi_{\boldsymbol{\nu}u'}(A_{p_{2}})\\ =\emptyset$,
 by the definition of $f$ and $z_{1},z_{2}\in\varphi_{\boldsymbol{\nu}}(K)$, we have
\begin{equation}\label{v}
\frac{\mid f(z_{1})-f(z_{2}) \mid}{\mid z_{1}-z_{2} \mid}
\leq\frac{\mathcal{H}^{s}(\varphi_{\boldsymbol{\nu}}(K)\bigcap E)}{\rho_{\boldsymbol{\nu}} \ \xi_{3}}
=\frac{\mathcal{H}^{s}(\varphi_{\boldsymbol{\nu}}(E))}{\rho_{\boldsymbol{\nu}} \ \xi_{3}}= M_{3} (\rho_{\boldsymbol{\nu}})^{s-1},
\end{equation}
where $\xi_{3}=\min\limits_{\substack{\ell\in\{1,\ldots,m+N\}\\  u,u'\in\{1,2,\ldots,N\}}}\{dist\big(\varphi_{u}(A_{m+N}),\varphi_{u'}(A_{\ell})\big); \varphi_{u}(A_{m+N})\cap\varphi_{u'}(A_{\ell})=\emptyset\}$ and $M_{3}=\frac{\mathcal{H}^{s}(E)}{\xi_{3}}$.

By symmetry, the case $\ell_{2} \! < \! \ell_{1}$ is similar as above.

Let $|z_{1}-z_{2}| \! \rightarrow \! 0$, then $|\boldsymbol{\nu}| \! \rightarrow \! \infty$ which lead to $\rho_{\boldsymbol{\nu}} \! \rightarrow \! 0$, it follows that $(\rho_{\boldsymbol{\nu}})^{s-1} \! \rightarrow \! 0$ by $s \! > \! 1$. Summing up with \eqref{q}-\eqref{v}, $| f(z_{1})-f(z_{2})| \! = \! o(| z_{1}-z_{2}|)$ holds for any $z_{1},z_{2}\in(K,\{a,b\})$, then $K$ is a Whitney set by Lemma \ref{lem-extention}. The theorem is proved.
\end{proof}

\end{document}